\documentclass[11pt,a4paper,reqno]{amsart}
\usepackage{graphics}
\usepackage{epsfig}

\usepackage{amsmath,amsthm,amssymb}
\newtheorem{theorem}{Theorem}[section]

\newtheorem{lemma}[theorem]{Lemma}

\newtheorem{example}[theorem]{Example}
\theoremstyle{definition}

\theoremstyle{remark}
\newtheorem{remark}[theorem]{Remark}
\numberwithin{equation}{section}

\begin{document}
\title[A SDFEM for system of two singularly...]{A SDFEM for system of two
singularly perturbed problems of convection-diffusion type with discontinuous source term.}
\author[]{\textbf{ A. ~Ramesh Babu }\quad  \\  Department of Mathematics \\  School of Engineering \\
AMRITA Vishwa Vidyapeetham \\  Coimbatore - 641 112 \\ Tamilnadu,  India.}
\thanks{The author wishes to
acknowledge a support rendered by Bharathidasan University,  Tamilnadu,  India through collabration with Prof. N. Ramanujam.
\\ E-mail: a\_rameshbabu@cb.amrita.edu,  matramesh2k5@yahoo.co.in(A. ~Ramesh Babu)}
\begin{abstract}
We consider a system of two singularly perturbed Boundary Value Problems (BVPs) of convection-diffusion type with discontinuous source terms and a small positive parameter multiplying the highest derivatives.  Then their solutions exhibit boundary layers as well as weak interior layers.  A numerical method based on finite element method (Shishkin and Bakhvalov-Shishkin meshes) is presented.  We derive an error estimate of order $O(N^{-1}\ln^{3/2}{N})$ in the energy norm with respect to the perturbation parameter.  Numerical experiments are also presented to support our theoritical results.
\vskip 0.5 true cm
\noindent {AMS Mathematics Subject Classification:}{ 65L10, CR G1.7}

\end{abstract}
\maketitle
\section{Introduction}
      Singularly Perturbed Differential Equations(SPDEs) appear in several
branches of applied mathematics.
 Analytical and numerical treatment
of these equations have drawn much attention of many researchers
\cite{mil2000, dool, roos96, nay, Omalley90}.  In general, classical
numerical methods fail to produce good approximations for these
equations. Hence one has to look for non-classical methods.  A good
number of articles have been appearing in the past three decades on
non-classical methods which cover mostly second order equations.
But only a few authors have developed numerical methods
for singularly perturbed system of ordinary
differential equations.\cite{linssmadden04,maddenstyn03,linsmad03,cen05,bellew04,tamram07a}.  \\
\quad Systems of this kind have applications in electro analytic chemistry when investigating diffusion processes complicated by chemical reactions.  The parameters multiplying the highest derivatives characterize the diffusion coefficient of the substances.  Other applications include equations of predator-prey population dynamics.  As was mentioned above,  classical numerical methods fails to produce good approximations for  singularly perturbed system of equations also.  Hence various methods are proposed in the literature in order to obtain numerical solution to singularly perturbed system of second order differential equations subject to Dirichlet type boundary conditions when the source terms are smooth on $(0,1)$ \cite{maddenstyn03,cen05,bellew04}.  Motivated by the works of T. Lin$\ss{}$ and N. Madden \cite{linssmadden04}, in the present paper we suggest a numerical method for singularly perturbed weakly coupled system of two ordinary differential equations of convection-diffusion type with discontinuous source terms.  Basically the method is based on Streamline Diffusion Finite Element Method (SDFEM) with layer adapted meshes like Shishkin and  Bakhvalov-Shishkin meshes.  For this method we derive an error estimate of order $O(N^{-1}\ln^{3/2}{N})$  in the energy norm.\\
In this paper,  we consider the system of singularly perturbed BVP with discontinuous source term
\begin{eqnarray}\label{sys1}
P_1\bar u:= -\varepsilon u_1''(x)+b_1(x)u_1'(x)+a_{11}(x)u_1(x)+a_{12}(x)u_2(x) = f_1(x),\quad x\in (\Omega ^-\cup\Omega^+)\\ \label{sys2}
P_2\bar u:= -\varepsilon u_2''(x)+b_2(x)u_2'(x)+a_{21}(x)u_1(x)+a_{22}(x)u_2(x) = f_2(x), \quad
x\in (\Omega ^-\cup\Omega^+) \\ \label{sysbc}
 u_1(0)=0,\quad u_1(1)=0,\quad
u_2(0)=0, \quad  u_2(1)=0,&
\end{eqnarray}
with the following conditions.
\begin{eqnarray}\label{sysc1}
b_1(x)\geq \beta_1 > 0,  \quad b_2(x) \geq \beta_2 > 0, \\ \label{sysc2}
 a_{12}(x) \leq 0, \quad a_{21}(x) \leq 0, \\ \label{sysc3}
a_{11}(x) > \vert a_{21}(x) \vert , \quad a_{22}(x) > \vert a_{12}(x) \vert , \quad \forall x \in \bar\Omega,
 \end{eqnarray}
 $ A = [a_{ij}], i=1,2; j=1,2 \,$  satisfies the property
 \begin{equation}\label{sysc4}
  \xi^T A \xi \geq \alpha  \xi \xi^T \quad \text{for every} \quad \xi=(\xi_1,\xi_2) \in \Re^2.
 \end{equation}
 For $ k=1,2 $
 \begin{equation}\label{sysc5}
\alpha -  \frac{1}{2} b_{k}' \geq \sigma_k ,\quad \text{for some} \quad \alpha, \sigma_k >0.
\end{equation}
where $\varepsilon >0$ is a small parameter,  $\Omega=(0,1),$ $\Omega^-=(0,d),$ $\Omega^+=(d,1),$ $d\in \Omega, $ and $u_1,u_2 \in U \equiv  C^0(\bar{\Omega})\cap C^1(\Omega)\cap C^2(\Omega^- \cup \Omega^+),$ $\bar u = (u_1,u_2)^T$.  Further it is assumed that the source terms $f_1,f_2$ are sufficiently smooth on $ \bar{\Omega} \setminus  \{d\};$  both the functions $f_1(x)$ and $f_2(x)$ are assumed to have a single discontinuity at the point $d\in \Omega.$   In general this discontinuity gives rise to interior layers in the solution of the problem.  Because $f_i,i=1,2$ are discontinuous at $d$   the solution $\bar u$ of (\ref{sys1}) - (\ref{sysbc}) does not necessarily have a continuous second derivative at the point $d.$  That is $u_1,u_2 \notin C^2(\Omega).$   But the first derivative of the solution exists and is continuous.
The authors from \cite{tamram07a} proved almost first order of convergence with respect to $\varepsilon$ on a Shishkin mesh of the finite difference method with special discretization in the point $d.$
\begin{remark}
Through out this paper,  $C,\,C_1$ denote  generic constants that are independent of the parameter $\varepsilon$ and $N,$  the dimension of the discrete problem.  We also assume $\varepsilon \leq CN^{-1}$ as is generally the case in practice.
\end{remark}
For our later analysis it is useful to have a decomposition of $\bar u$ in the smooth part $\bar v$ and the layer part $\bar w.$  That is
\begin{equation*}
\bar u = \bar v + \bar w_1 + \bar w_2, \quad  \text{where}  \quad \bar v=(v_1,v_2), \quad \bar w_1=(w_{11},w_{12}), \quad \bar w_2=(w_{21},w_{22}).
\end{equation*}
\begin{theorem}
With the decomposition of the above,  for each $k,$ $0 \leq k \leq 3,$ and $j=1,2$ it holds
\begin{eqnarray*}
\vert v_{j}^{(k)}(x) \vert \leq C(1+ \varepsilon^{(2-k)}), \quad x \in \Omega,\\
\vert w_{1j}^{(k)}(x) \vert \leq C \varepsilon^{-k}e^{\frac{-\beta (1-x)}{\varepsilon}}, \quad x \in \bar \Omega,\\
\vert w_{2j}^{(k)}(x) \vert \leq \begin{cases}
C \varepsilon^{(1-k)}e^{\frac{- \beta (d-x)}{\varepsilon}}, \quad x \in \Omega^-,\\
C \varepsilon^{(1-k)}e^{\frac{-\beta(1-x)}{\varepsilon}}, \quad x \in \Omega^+,
                               \end{cases}
\end{eqnarray*}
where $\beta = \min \{ \beta_1,\beta_2 \}.$
\end{theorem}
\begin{proof}
Using the results of \cite{linsmad03} and adopting the technique of \cite{mil2000} this theorem can be proved.
\end{proof}
This paper is organized as follows.  Section $2$ presents a weak formulation of the BVP (\ref{sys1}) - (\ref{sysbc}).   We define an energy norm on $(H_0^1(\Omega))^2$ and  discribe a finite element discretization of the problem.  Section $3$  presents an analysis of the  corresponding scheme on Shishkin and Bakhvalov-Shishkin meshes.  In section $4,$  we present  an interpolation error on various norms.  The paper concludes with numerical examples.
\section{Analytical results}
A standard weak formulation of (\ref{sys1})-(\ref{sysbc}) is: Find $u_1, u_2 \in H_0^1(\Omega)$ such that
\begin{eqnarray}\label{wsys1}
B_1(u_1,v_1) = f_1(v_1),\quad  \forall v_1 \in H_0^1(\Omega)\\ \label{wsys2}
B_2(u_2,v_2) = f_2(v_2),\quad  \forall v_2 \in H_0^1(\Omega)
\end{eqnarray}
where
\begin{eqnarray*}
B_1(u_1,v_1):= \varepsilon(u_1',v_1')+(b_1u_1',v_1)+(a_{11}u_1+a_{12}u_2,v_1),\\
B_2(u_2,v_2):= \varepsilon(u_2',v_2')+(b_2u_2',v_2)+(a_{21}u_1+a_{22}u_2,v_2)
\end{eqnarray*}
and
\begin{eqnarray*}
f_1(v_1) = (f_1,v_1),\\
f_2(v_2) = (f_2,v_2).
\end{eqnarray*}
Here $H_0^1(\Omega)$ denotes the usual Sobolev space and $(.,.)$ is the inner product on $L_2(\Omega).$  Now we combine the two equations (\ref{wsys1}) - (\ref{wsys2}) and get a single weak formulation.  Then our problem is:  Find $\bar u \in (H_0^1(\Omega))^2 $ such that
\begin{equation}\label{wprob}
B(\bar u, \bar v) = f(\bar v), \quad  \forall \bar v \in (H_0^1(\Omega))^2
\end{equation}
with  $B(\bar u, \bar v) := B_1(u_1,v_1)+B_2(u_2,v_2) \quad$ and $\quad
f(\bar v) := f_1(v_1)+f_2(v_2).$
Now we define a norm on $(H_0^1(\Omega))^2 $ associated with the bilinear form $B(.,.)$,  called continuous energy norm as
$\vert\vert\vert \bar u \vert\vert\vert_{H_0^1} = [\varepsilon (\vert u_1 \vert_1^2 + \vert u_2 \vert_1^2) +\sigma(\Vert u_1 \Vert_0^2 + \Vert u_2 \Vert_0^2)]^{1/2},$
where $\sigma = \min\{\sigma_1,\sigma_2\} $ and  $ \Vert u \Vert_0 :=(u,u)^{1/2}$ is the standard norm on $L_2(\Omega),$  while
$\vert u \vert_1 := \Vert u' \Vert_0 $ is the usual semi-norm on $ H_0^1(\Omega). $  We also use the notation $\Vert \bar u \Vert_0 = (\Vert u_1 \Vert_0^2 + \Vert u_2 \Vert_0^2)^{1/2}$ for the norm in $ (L_2(\Omega))^2.$\\
$B$ is a bilinear functional defined on $(H_0^1(\Omega))^2.$  Further we have to prove that it is coercive with respect to $\vert\vert\vert . \vert\vert\vert_{H_0^1}, $  that is
 $ B(\bar u, \bar u) \geq \vert\vert\vert \bar u \vert\vert\vert_{H_0^1}^2.$
 \begin{lemma}\label{coerlem}
A bilinear functional $B$ satisfies the coercive property with respect to $\vert\vert\vert . \vert\vert\vert_{H_0^1}.$
 \end{lemma}
 \begin{proof}
 Let $\bar u = (u_1,u_2) \in (H_0^1(\Omega))^2.$  Then
\begin{eqnarray*}
B(\bar u, \bar u) & = & \varepsilon(u_1',u_1')+(b_1u_1',u_1) +(a_{11}u_1+a_{12}u_2,u_1) +\varepsilon(u_2',u_2') + (b_2u_2',u_2)\\
&  & +(a_{21}u_1+a_{22}u_2,u_2)\\
& \geq & \varepsilon(\vert u_1 \vert_1^2 + \vert u_2 \vert_1^2)+\int_0^1 b_1(x)u_1'u_1 dx
 + \int_0^1 b_2(x)u_2'u_2 dx  + (\alpha u_1,u_1)\\
 &  & +(\alpha u_2,u_2)\\
& = & \varepsilon(\vert u_1 \vert_1^2 + \vert u_2 \vert_1^2)+ \int_0^1 \frac{b_1(x)}{2}\frac{d}{dx}(u_1^2)+ \int_0^1 \alpha u_1^2 dx
+\int_0^1 \frac{b_2(x)}{2}\frac{d}{dx}(u_2^2)\\
&   & + \int_0^1 \alpha u_2^2 dx\\
& = & \varepsilon(\vert u_1 \vert_1^2 + \vert u_2 \vert_1^2) - \frac{1}{2} \int_0^1 u_1^2 d(b_1(x))+ \int_0^1 \alpha u_1^2 dx
 -  \frac{1}{2} \int_0^1 u_2^2 d(b_2(x)) \\
 &  & + \int_0^1 \alpha u_2^2 dx
 \end{eqnarray*}
\begin{eqnarray*}
& = & \varepsilon(\vert u_1 \vert_1^2 + \vert u_2 \vert_1^2) + \int_0^1(\alpha - \frac{1}{2} b_1'(x))u_1^2 dx + \int_0^1(\alpha - \frac{1}{2} b_2'(x))u_2^2 dx\\
& \geq &  \varepsilon(\vert u_1 \vert_1^2 + \vert u_2 \vert_1^2) + \min\{\sigma_1,\sigma_2\} [\int_0^1  u_1^2 dx + \int_0^1  u_1^2 dx]\\
B(\bar u, \bar u) & \geq &  \varepsilon(\vert u_1 \vert_1^2 + \vert u_2 \vert_1^2) + \sigma(\Vert u_1 \Vert_0^2 + \Vert u_2 \Vert_0^2)\\
 \end{eqnarray*}
\quad Therefore we have
\begin{equation*}
B(\bar u, \bar u)  \geq  \vert\vert\vert \bar u \vert\vert\vert^2.
\end{equation*}
Hence $B$ is coercive with respect to $\vert\vert\vert . \vert\vert\vert.$
 \end{proof}
Also $B$ is continuous in the energy norm and $f$ is a bounded linear functional on $(H_0^1(\Omega))^2.$
 By Lax-Milgram Theorem,  we conclude that the problem (\ref{wprob}) has a unique solution.
\subsection{Discretization of weak problem}
 Let $\Omega_\varepsilon^N=\{x_0,x_1,\cdots, x_N\}$ to be the set of mesh
 points $x_i$, for some positive integer $N$. For
 $i\in \{1,2,\cdots,N\}.$  We set $h_i=x_i-x_{i-1}$ to be the local
mesh step size, and for $i\in \{1,2,\cdots,N\}$\, let
$\bar{h_i}=(h_i+h_{i+1})/2$.
Let $V_h \subset H_0^1(\Omega)$ be the space of piecewise linear functions on $\Omega$.
As usual, basis functions of $V_h$ are given by
 \begin{equation*}
   \phi_i(x) =
   \begin{cases}
   \frac{x-x_{i-1}}{h_i}, \quad x\in [x_{i-1},x_i]\\
   \frac{x_{i+1}-x}{h_{i+1}}, \quad x\in [x_{i},x_{i+1}]\\
   0, \quad x \notin [x_{i-1},x_{i+1}].
 \end{cases}
   \end{equation*}
Then our discretization of (\ref{wprob}) is: Find $\bar u_h \in V_h^2$ such that
\begin{equation}\label{dissys}
B_h(\bar u_h, \bar v_h) = f_h(\bar v_h), \quad  \forall \bar v_h \in V_h^2,
\end{equation}
   where
\begin{align*}
B_h(\bar u_h, \bar v_h):= & (\varepsilon u_{1h}', v_{1h}')+(b_1u_{1h}',v_{1h})+(a_{11}u_{1h}+a_{12}u_{2h},v_{1h}) +(\varepsilon u_{2h}', v_{2h}')\\
& +(b_2u_{2h}',v_{2h}) + (a_{21}u_{1h}+a_{22}u_{2h},v_{2h}) \\
& + \sum_{i=1}^{N}\int_{x_{i-1}}^{x_{i}}\delta_{1,i}(-\varepsilon u_{1h}''(x)+b_1(x)u_{1h}'(x)+a_{11}(x)u_{1h}(x)+a_{12}(x)u_{2h}(x))b_1v_{1h}'dx\\
&+\sum_{i=1}^{N}\int_{x_{i-1}}^{x_{i}}\delta_{2,i}(-\varepsilon u_{2h}''(x)+b_2(x)u_{2h}'(x)+a_{21}(x)u_{1h}(x)+a_{22}(x)u_{2h}(x))b_2v_{2h}'dx\\
f_h(\bar v_h):=&(f_1,v_{1h})+(f_2,v_{2h})+\sum_{i=1}^{N}\int_{x_{i-1}}^{x_{i}}\delta_{1,i}f_1b_1v_{1h}'+\sum_{i=1}^{N}\int_{x_{i-1}}^{x_{i}}\delta_{2,i}f_2b_2v_{2h}'dx.
\end{align*}
The parameters $\delta_{1,i} \geq 0$ and $\delta_{2,i} \geq 0$ are called the streamline-diffusion parameters and will be determined later.
Here we define a discrete energy norm on $V_h^2$ associated with the bilinear form $B_h(.,.)$ as
\begin{eqnarray*}
\vert\vert\vert \bar u_h \vert\vert\vert_{V_h} & = &[\varepsilon (\vert u_{1h} \vert_1^2 + \vert u_{2h} \vert_1^2) +\sigma(\Vert u_{1h} \Vert_0^2 + \Vert u_{2h} \Vert_0^2)
+\sum_{i=1}^{N}\int_{x_{i-1}}^{x_{i}}\delta_{1,i}b_1^2(x_i)(u_{1h}'(x))^2 dx\\
&  & +\sum_{i=1}^{N}\int_{x_{i-1}}^{x_{i}}\delta_{2,i}b_2^2(x_i)(u_{2h}'(x))^2 dx]^{1/2}.
\end{eqnarray*}
$B_h$ is a bilinear functional defined on $V_h^2.$   Further we have to prove that it is coercive with respect to $\vert\vert\vert . \vert\vert\vert_{V_h},$   that is
 $ B_h(\bar u_h, \bar u_h) \geq \vert\vert\vert \bar u_h \vert\vert\vert_{V_h}^2.$
\begin{lemma}\label{coerlem1}
If $\quad \delta_{1,i} = \delta_{2,i} = 0\,\,\quad $ then $\quad  B_h(\bar u_h, \bar u_h) \geq \vert\vert\vert \bar u_h \vert\vert\vert_{V_h}^2 $\\ and
if $\quad  0 < \delta_{1,i}, \delta_{2,i} \leq \frac{1}{4} \min_{i=1,2} \{ \frac{\sigma_i}{\mu^2}\},\quad \mu = \max_{x \in \bar \Omega} \{ \mid a_{ij}(x) \mid \}, i,j=1,2\quad$  then $ \quad B_h(\bar u_h, \bar u_h) \geq \frac{1}{2}\vert\vert\vert \bar u_h \vert\vert\vert_{V_h}^2. $  That is,  a  bilinear functional $B_h$ satisfies the coercive property with respect to $\vert\vert\vert . \vert\vert\vert_{V_h}.$
 \end{lemma}
 \begin{proof}
Let $\bar u_h = (u_{1h},u_{2h}) \in V_h^2.$  If $\quad \delta_{1,i} = \delta_{2,i} = 0 \quad $ then the result directly follows from Lemma (\ref{coerlem}).\\
If  $\quad  0 < \delta_{1,i}, \delta_{2,i} \leq \frac{1}{4} \min_{i=1,2} \{ \frac{\sigma_i}{\mu^2}\}\quad $  then we have
\begin{eqnarray*}
B_h(\bar u_h, \bar u_h)& = & \varepsilon(u_{1h}',u_{1h}') +(b_1u_{1h}',u_{1h})+(a_{11}u_{1h}+a_{12}u_{2h},u_{1h})    +\varepsilon(u_{2h}',u_{2h}') +(b_2u_{2h}',u_{2h}) \\
&  & + (a_{21}u_{1h}+a_{22}u_{2h},u_{2h}) +\sum_{i=1}^{N} \int_{x_{i-1}}^{x_i} \delta_{1,i}(-\varepsilon u_{1h}''+b_1u_{1h}'+a_{11}u_{1h}+a_{12}u_{2h}) b_1u_{1h}'dx \\
&  & +\sum_{i=1}^{N} \int_{x_{i-1}}^{x_i} \delta_{2,i}(-\varepsilon u_{2h}''+b_2u_{2h}'+a_{21}u_{1h}+a_{22}u_{2h}) b_2u_{2h}'dx\\
& \geq & \varepsilon(\vert u_{1h} \vert_1^2 + \vert u_{2h} \vert_1^2)+\int_0^1 b_1(x)u_{1h}'u_{1h} dx
 + \int_0^1 b_2(x)u_{2h}'u_{2h} dx  + \int_0^1 \alpha u_{1h}^2 dx\\
 &  & +\int_0^1 \alpha u_{2h}^2 dx +\sum_{i=1}^{N} \int_{x_{i-1}}^{x_i} \delta_{1,i}(b_1^2(u_{1h}')^2)dx \\
&  & + \sum_{i=1}^{N} \int_{x_{i-1}}^{x_i} \delta_{1,i}(a_{11}u_{1h}+a_{12}u_{2h}) b_1u_{1h}'dx + \sum_{i=1}^{N} \int_{x_{i-1}}^{x_i} \delta_{2,i}(b_2^2(u_{2h}')^2)dx  \\
&  & +  \sum_{i=1}^{N} \int_{x_{i-1}}^{x_i} \delta_{2,i}(a_{21}u_{1h}+a_{22}u_{2h}) b_2u_{2h}'dx\\
& \geq & \varepsilon(\vert u_{1h} \vert_1^2 + \vert u_{2h} \vert_1^2) + \int_0^1(\alpha - \frac{1}{2} b_1'(x))u_{1h}^2 dx + \int_0^1(\alpha - \frac{1}{2} b_2'(x))u_{2h}^2 dx\\
& & + \sum_{i=1}^{N} \int_{x_{i-1}}^{x_i} \delta_{1,i}(b_1^2(u_{1h}')^2)dx + \sum_{i=1}^{N} \int_{x_{i-1}}^{x_i} \delta_{2,i}(b_2^2(u_{2h}')^2)dx \\
&  & +\sum_{i=1}^{N} \int_{x_{i-1}}^{x_i} \delta_{1,i}(a_{11}u_{1h}+a_{12}u_{2h}) b_1u_{1h}'dx + \sum_{i=1}^{N} \int_{x_{i-1}}^{x_i} \delta_{2,i}(a_{21}u_{1h}+a_{22}u_{2h}) b_2u_{2h}'dx\\
B_h(\bar u_h, \bar u_h) & \geq &  \varepsilon(\vert u_{1h} \vert_1^2 + \vert u_{2h} \vert_1^2) + \sigma(\Vert u_{1h} \Vert_0^2 + \Vert u_{2h} \Vert_0^2) \\
&  & + \sum_{i=1}^{N} \int_{x_{i-1}}^{x_i} \delta_{1,i}(b_1^2(u_{1h}')^2)dx + \sum_{i=1}^{N} \int_{x_{i-1}}^{x_i} \delta_{2,i}(b_2^2(u_{2h}')^2)dx \\
&  & +\sum_{i=1}^{N} \int_{x_{i-1}}^{x_i} \delta_{1,i}(a_{11}u_{1h}+a_{12}u_{2h}) b_1u_{1h}'dx + \sum_{i=1}^{N} \int_{x_{i-1}}^{x_i} \delta_{2,i}(a_{21}u_{1h}+a_{22}u_{2h}) b_2u_{2h}'dx
 \end{eqnarray*}
 Using the assumption on $\delta_{1,i}$ and $\delta_{2,i},$  we obtain
 \begin{eqnarray*}
&  & \mid  \sum_{i=1}^{N} \int_{x_{i-1}}^{x_i} \delta_{1,i}(a_{11}u_{1h}+a_{12}u_{2h}) b_1u_{1h}'dx \mid \\ & \leq  &  \sum_{i=1}^{N} \int_{x_{i-1}}^{x_i} \delta_{1,i} \mid  a_{11}u_{1h}\mid ^2 dx + \sum_{i=1}^{N} \int_{x_{i-1}}^{x_i} \delta_{1,i}\mid a_{12}u_{2h} \mid^2 dx
 + \frac{1}{2} \sum_{i=1}^{N} \int_{x_{i-1}}^{x_i} \delta_{1,i} \mid b_1u_{1h}'\mid^2 dx \\
& \leq & \sum_{i=1}^{N} \int_{x_{i-1}}^{x_i} (\frac{\sigma}{4\mu^2}) \mu^2 \mid u_{1h}\mid ^2 dx + \sum_{i=1}^{N} \int_{x_{i-1}}^{x_i} (\frac{\sigma}{4\mu^2}) \mu^2\mid u_{2h} \mid^2 dx
+ \frac{1}{2} \sum_{i=1}^{N} \int_{x_{i-1}}^{x_i} \delta_{1,i} \mid b_1u_{1h}'\mid^2 dx \\
& = & \frac{\sigma}{4}(\Vert u_{1h} \Vert_0^2 + \Vert u_{2h} \Vert_0^2)+ \frac{1}{2} \sum_{i=1}^{N} \int_{x_{i-1}}^{x_i} \delta_{1,i}  (b_1 u_{1h}')^2 dx
\end{eqnarray*}
and similarly we have
\begin{equation*}
\mid  \sum_{i=1}^{N} \int_{x_{i-1}}^{x_i} \delta_{2,i}(a_{21}u_{1h}+a_{22}u_{2h}) b_2u_{2h}'dx \mid   \leq   \frac{\sigma}{4}(\Vert u_{1h} \Vert_0^2 + \Vert u_{2h} \Vert_0^2)+ \frac{1}{2} \sum_{i=1}^{N} \int_{x_{i-1}}^{x_i} \delta_{2,i}  (b_2 u_{1h}')^2 dx.
 \end{equation*}
Combining the above two results  we have the desired result.
Hence $B_h$ is coercive with respect to $\vert\vert\vert . \vert\vert\vert_{V_h}.$
 \end{proof}
Also $B_h$ is continuous in the discrete energy norm and $f_h$ is a bounded linear functional on $V_h^2.$  By Lax-Milgram Theorem,  we conclude that the problem (\ref{dissys}) has a unique solution.\\
\begin{remark}
While deriving the corresponding difference scheme,  we use the SDFEM with lumping for the terms $ (a_{11}u_1+a_{12}u_2,v_1)$ and $(a_{21}u_1+a_{22}u_2,v_2).$  That is $(a_{11}u_1,v_1)$ is replaced by $\sum_{i=1}^{N-1}\bar h_i\widehat{a_{11,i}}u_{1,i}v_{1,i}$  where  $\widehat{a_{11,i}} = \frac{\bar b_1^2}{\beta_1^2} \Vert a_{11}\Vert_{\infty[x_i,x_{i-1}]}.$
\end{remark}
We choose $d = x_{N/2}$ and take $f_1(d) = f_1(N/2)=\frac{f_1(\frac{N}{2}-1)+f_1(\frac{N}{2}+1)}{2},\, \quad $ $f_2(d) = f_2(N/2)=\frac{f_2(\frac{N}{2}-1)+f_2(\frac{N}{2}+1)}{2}.$
Then the corresponding difference scheme is
\begin{equation}\label{difsch}
\begin{split}
L^N \bar U_{i} &:=
\begin{cases}
-\varepsilon[(\frac{U_{1,i+1}-U_{1,i}}{h_{i+1}}-\frac{U_{1,i}-U_{1,i-1}}{h_{i}})+(\frac{U_{2,i+1}-U_{2,i}}{h_{i+1}}-\frac{U_{2,i}-U_{2,i-1}}{h_{i}})]\\
+\alpha_{1,i}(\frac{U_{1,i+1}-U_{1,i}}{h_{i+1}})+\alpha_{2,i}(\frac{U_{2,i+1}-U_{2,i}}{h_{i+1}})\\
+\beta_{1,i}(\frac{U_{1,i}-U_{1,i-1}}{h_{i}})+\beta_{2,i}(\frac{U_{2,i}-U_{2,i-1}}{h_{i}})\\
+\gamma_{1,i}U_{1,i}+\gamma_{2,i}U_{2,i}= f_h(\bar \phi_i),
\end{cases}\\
&U_{1,0} = U_{1,N} = U_{2,0} = U_{2,N} = 0,
\end{split}
\end{equation}
where $\bar U_{i}=(U_{1,i},U_{2,i}),\,$ $U_{1,i}=U_1(x_i),\, U_{2,i}=U_2(x_i),\,\bar \phi_i = (\phi_i,\phi_i), \quad
i=1,2,...,N-1$ and
\begin{align*}
\alpha_{1,i}&= h_{i+1} \int_{x_{i}}^{x_{i+1}} (b_1\phi_{i+1}'\phi_{i}+\delta_{1,i+1}b_1^2\phi_{i+1}'\phi_{i}'+\delta_{1,i+1}b_1a_{11}\phi_{i+1}\phi_{i}'+\delta_{2,i+1}b_2a_{21}\phi_{i+1}\phi_{i}')dx\\
\beta_{1,i}&=-h_{i}\int_{x_{i-1}}^{x_{i}}(b_1\phi_{i-1}'\phi_{i}+\delta_{1,i}b_1^2\phi_{i-1}'\phi_{i}'+\delta_{1,i}b_1a_{11}\phi_{i-1}\phi_{i}'+\delta_{2,i}b_2a_{21}\phi_{i-1}\phi_{i}')dx\\
\gamma_{1,i}&=\bar h_{i}(\widehat{a_{11}}+\widehat{a_{21}})(x_i) +\int_{x_{i-1}}^{x_{i}}(\delta_{1,i}b_1a_{11}+\delta_{2,i}b_2a_{21})\phi_{i}'dx +\int_{x_{i}}^{x_{i+1}}(\delta_{1,i+1}b_1a_{11}+\delta_{2,i+1}b_2a_{21})\phi_{i}'dx\\
\alpha_{2,i}&= h_{i+1} \int_{x_{i}}^{x_{i+1}} (b_2\phi_{i+1}'\phi_{i}+\delta_{2,i+1}b_2^2\phi_{i+1}'\phi_{i}'+\delta_{1,i+1}b_1a_{12}\phi_{i+1}\phi_{i}'+\delta_{2,i+1}b_2a_{22}\phi_{i+1}\phi_{i}')dx\\
\beta_{2,i}&=-h_{i}\int_{x_{i}}^{x_{i+1}}(b_2\phi_{i-1}'\phi_{i}+\delta_{2,i}b_2^2\phi_{i-1}'\phi_{i}'+\delta_{1,i}b_1a_{12}\phi_{i-1}\phi_{i}'+\delta_{2,i}b_2a_{22}\phi_{i-1}\phi_{i}')dx\\
\gamma_{2,i}&=\bar h_{i}(\widehat{a_{12}}+\widehat{a_{22}})(x_i) +\int_{x_{i-1}}^{x_{i}}(\delta_{1,i}b_1a_{12}+\delta_{2,i}b_2a_{22})\phi_{i}'dx +\int_{x_{i}}^{x_{i+1}}(\delta_{1,i+1}b_1a_{12}+\delta_{2,i+1}b_2a_{22})\phi_{i}'dx.
\end{align*}
\begin{remark}
If the local  mesh step is small enough,  then it is possible to choose $\delta_{k,i} = 0, k=1,2.$  In other case,  we shall choose $\delta_{k,i}$ from the condition, $\alpha_{k,i}$ of the difference scheme (\ref{difsch}) equal to zero.    Thus we have
\begin{align*}
\delta_{1,i} = \begin{cases}
0, \quad \quad \quad \quad \quad \quad \quad h_i \leq \frac{2\varepsilon}{\parallel b_{1} \parallel_\infty},\\
\frac{b_1h_i(2b_2^2 +h_ib_2a_{22})-h_i^2b_2^2a_{21}}{(2b_1^2 +h_ib_1a_{11})(2b_2^2 +h_ib_2a_{22})-h_i^2b_1b_2a_{12}a_{21}}, h_i > \frac{2\varepsilon}{\parallel b_{1} \parallel_\infty}
               \end{cases}
\end{align*}
and also
\begin{align*}
\delta_{2,i} = \begin{cases}
0, \quad \quad \quad \quad \quad \quad \quad h_i \leq \frac{2\varepsilon}{\parallel b_{2} \parallel_\infty},\\
\frac{b_2h_i(2b_1^2 +h_ib_1a_{11})-h_i^2b_1^2a_{12}}{(2b_1^2 +h_ib_1a_{11})(2b_2^2 +h_ib_2a_{22})-h_i^2b_1b_2a_{12}a_{21}}, h_i > \frac{2\varepsilon}{\parallel b_{2} \parallel_\infty}.
               \end{cases}
\end{align*}
We derive the following estimates of $\delta_{1,i}$ and $\delta_{2,i}$
\begin{eqnarray*}
\delta_{k,i} \leq \begin{cases}
C N^{-1} \quad \text{for} \quad i=1,...,N/4 \quad\text{and}\quad i=(N/2)+1,...,3N/4,\\
0   \quad \quad \quad\text{for} \quad i=(N/4)+1,...,N/2 \quad\text{and}\quad i=(3N/4)+1,...,N-1,
                  \end{cases}
\end{eqnarray*}
where $k=1,2.$
\end{remark}

The above system  contains $N-1$ equations and has $2N-2$ unknowns.  To solve the system we split it into two algebraic systems as follows: \\
For $i=1,2,...,N-1$
\begin{align}\label{rsys1}
P_1^{N}U_{1,i}^{*}:=\begin{cases} -\varepsilon(\frac{U_{1,i+1}^*-U_{1,i}^*}{h_{i+1}}-\frac{U_{1,i}^*-U_{1,i-1}^*}{h_{i}})+\alpha_{1,i}(\frac{U_{1,i+1}^*-U_{1,i}^*}{h_{i+1}})+\beta_{1,i}(\frac{U_{1,i}^*-U_{1,i-1}^*}{h_{i}})\\
+\gamma_{1,i}U_{1,i}^*=\int_{x_{i-1}}^{x_{i+1}}f_1\phi_i+\sum_{i=1}^{N}\int_{x_{i-1}}^{x_{i}}\delta_{1,i}f_1b_1\phi_i',\quad U_{1,0}^*=U_{1,N}^*=0,
\end{cases}
\end{align}
\begin{align} \label{rsys2}
P_2^{N}U_{2,i}^{*}:=\begin{cases}
-\varepsilon(\frac{U_{2,i+1}^*-U_{2,i}^*}{h_{i+1}}-\frac{U_{2,i}^*-U_{2,i-1}^*}{h_{i}})+\alpha_{2,i}(\frac{U_{2,i+1}^*-U_{2,i}^*}{h_{i+1}})+\beta_{2,i}(\frac{U_{2,i}^*-U_{2,i-1}^*}{h_{i}})\\
+\gamma_{2,i}U_{2,i}^*=\int_{x_{i-1}}^{x_{i+1}}f_2\phi_i+\sum_{i=1}^{N}\int_{x_{i-1}}^{x_{i}}\delta_{2,i}f_2b_2\phi_i', \quad U_{2,0}^*=U_{2,N}^*=0.
\end{cases}
\end{align}
The above system (\ref{rsys1}) corresponds to the differential equation
 \begin{equation*}
P_1^* u_1^*:= -\varepsilon u_1^{*''}+ b_1(x)u_1^{*'}+(a_{11}(x)+a_{21}(x))u_1^* = f_1(x),\quad x \in (\Omega^- \cup \Omega^+),
 \end{equation*}
subject to boundary conditions $ u_1^*(0) = u_1^*(1) = 0.$  This boundary value problem has a unique solution \cite{rszar2004a}.  Using the inverse monotone property of the matrix,  one can establish the numerical stability of the system (\ref{rsys1}).  Similarly we can deal with second equation (\ref{rsys2}).  If $U_{1,i}^*$ and $U_{2,i}^*$ are solutions of (\ref{rsys1}) and   (\ref{rsys2}) respectively then $(U_{1,i}^*,U_{2,i}^*)$ is a solution of (\ref{difsch}).  By uniqueness,  this is the only possible solution.  Therefore, it is enough to solve (\ref{rsys1}) and (\ref{rsys2}).
   \section{Error analysis}
   The convergence analysis of the numerical scheme starts at the triangle inequality
   \begin{equation}\label{err1}
\vert\vert\vert \bar u-\bar u_h \vert\vert\vert_{V_h} \leq \vert\vert\vert \bar u-\bar u^I \vert\vert\vert_{V_h} + \vert\vert\vert \bar u^I - \bar u_h \vert\vert\vert_{V_h},
   \end{equation}
   where $\bar u^I$ denotes the piecewise linear interpolant to $\bar u$ on $\Omega$.\\
   Now we estimate the second term of equation (\ref{err1}).
   \begin{lemma}\label{intpl2}
   The following estimate holds true
 \begin{equation*}
\vert\vert\vert \bar u^I - \bar u_h \vert\vert\vert_{V_h} \leq C \Vert \bar u^I- \bar u \Vert_0.
\end{equation*}
   \end{lemma}
    \begin{proof}
    Because of the Galerkin orthogonality relation between $\bar u$ and $\bar u_h$,
   we have
\begin{equation*}
B_h(\bar u_h-\bar u,\bar u^I - \bar u_h)  =  0.
\end{equation*}
    Then from the coercive property (\ref{coerlem1}) of $B_h(.,.),$  we have
\begin{eqnarray*}
\vert\vert\vert \bar u^I - \bar u_h \vert\vert\vert_{V_h}^2  &\leq& 2 B_h(\bar u^I - \bar u_h,\bar u^I - \bar u_h)\\ & =& 2 B_h(\bar u^I -\bar u ,\bar u^I - \bar u_h)\\
& = & 2[(b_1(u_1^I-u_1)',u_1^I-u_{1h})+(a_{11}(u_1^I-u_1),u_1^I-u_{1h}) +(a_{12}(u_2^I-u_2),u_1^I-u_{1h}) \\
&&+(b_2(u_2^I-u_2)',u_2^I-u_{2h})+ (a_{21}(u_1^I-u_1),u_2^I-u_{2h}) + (a_{22}(u_2^I-u_2),u_2^I-u_{2h})\\ &&+\sum_{i=1}^{N}\int_{x_{i-1}}^{x_i}\delta_{1,i}(-\varepsilon (u_1^I-u_1)''+b_1(u_1^I-u_1)'+a_{11}(u_1^I-u_1)\\
&&+a_{12}(x)(u_2^I-u_2))b_1(u_1^I-u_{1h})'dx\\
&& +\sum_{i=1}^{N}\int_{x_{i-1}}^{x_i}\delta_{2,i}(-\varepsilon (u_2^I-u_2)''+b_2(u_2^I-u_2)'+a_{21}(u_1^I-u_1)\\
&&+a_{22}(u_2^I-u_2))b_2(u_2^I-u_{2h})'dx].
   \end{eqnarray*}
   That is,
   \begin{eqnarray*}
\vert\vert\vert \bar u^I - \bar u_h \vert\vert\vert_{V_h}^2 &\leq & C \int_0^1 (u_1^I-u_1)[(u_1^I-u_{1h})+(u_2^I-u_{2h})] + C \int_0^1(u_2^I-u_2)[(u_1^I-u_{1h})+(u_2^I-u_{2h})] \\
 &\leq&  C \int_0^1 [(u_1^I-u_1)+(u_2^I-u_2)][(u_1^I-u_{1h})+(u_2^I-u_{2h})].
 \end{eqnarray*}
 Therefore we have
 \begin{eqnarray*}
\vert\vert\vert \bar u^I - \bar u_h \vert\vert\vert_{V_h}^2  &\leq & C \Vert \bar u^I- \bar u \Vert_0 \quad \Vert \bar u^I- \bar u_h \Vert_0 \\
\vert\vert\vert \bar u^I - \bar u_h \vert\vert\vert_{V_h}^2   &\leq&  C \Vert \bar u^I- \bar u \Vert_0 \quad  \vert\vert\vert \bar u^I- \bar u_h \vert\vert\vert_{V_h} \\
\vert\vert\vert \bar u^I - \bar u_h \vert\vert\vert_{V_h} &\leq & C \Vert \bar u^I- \bar u \Vert_0.
\end{eqnarray*}
\end{proof}
\subsection{Error analysis on Shishkin and Bakhvalov-Shishkin meshes}
For the discretization described above  we shall use a mesh of the general type introduced in \cite{rsl99},  but here adapted for the layers at $x=d.$
 Let $N>4$ be a positive even
integer and
\begin{equation*}
\sigma_1 = \min\{\frac{d}{2},{\frac{\varepsilon}{\beta}}\tau_0 \ln N\},\quad \sigma_2 =
\min\{\frac{1-d}{2},{\frac{\varepsilon}{\beta}}\tau_0 \ln N\},\quad \quad \tau_0 \geq 2.
\end{equation*}
  Our mesh will be equidistant on $\bar \Omega_S $,
where
\begin{equation*}
\Omega_S=(0,d-\sigma_1)\cup(d,1-\sigma_2)
\end{equation*}
and graded on $\bar\Omega_0$ where
\begin{equation*}
\Omega_0=(d-\sigma_1,d)\cup(1-\sigma_2,1).
\end{equation*}
First we shall assume $ \sigma_1 = \sigma_2 = \frac{\tau_0 \varepsilon}{\beta} \ln N $ as otherwise $N^{-1}$ is exponentially small compared to $ \varepsilon.$
We choose the transition points to be
\begin{equation*}
x_{N/4}=d-\sigma_1,\quad  x_{N/2}=d,\quad x_{3N/4}=1-\sigma_2.
\end{equation*}
Because of the specific layers, here we have to use two mesh generating
functions $\varphi_1$ and $\varphi_2$ which are both
continuous and piecewise continuously differentiable and  monotonically decreasing
functions and
\begin{eqnarray*}
%\end{eqnarray*}
\varphi_1(1/4)=\ln N,  &\quad\quad&  \varphi_1(1/2)=0\\
\varphi_2(3/4)=\ln N,  &\quad\quad&  \varphi_2(1)=0.
\end{eqnarray*}
 The mesh points are
 \begin{equation*}
 x_i=
\begin{cases}
\frac{4i}{N}(d-\sigma_1), \quad \quad \quad \quad \quad \quad\quad \quad\quad\quad \quad i=0,...,N/4\\
d-\frac{\tau_0}{\beta}\varepsilon \varphi_1(t_i),\quad \quad \quad\quad \quad\quad\quad\quad \quad i=N/4+1,...,N/2\\
d+\frac{4}{N}(1-d-\sigma_2)(i-N/2),\quad\quad  \quad i=N/2+1,...,3N/4\\
1-\frac{\tau_0}{\beta}\varepsilon \varphi_2(t_i),\quad\quad\quad \quad \quad \quad\quad\quad \quad\quad i=3N/4+1,...,N,
 \end{cases}
\end{equation*}
 where $t_i=i/N$.
 We define new functions $\psi_1$ and $\psi_2$
by
\begin{equation*}
\varphi_i = -\ln \psi_i, \quad   i=1,2.
\end{equation*}
There are several mesh-characterizing functions $\psi$ in the literature,  but we shall
use only those which  correspond to Shishkin mesh and
Bakhvalov-Shishkin mesh with the following properties
\begin{eqnarray*}
\max \vert \psi' \vert & = & C \ln N \quad \text{for Shishkin meshes}\\
\max \vert \psi' \vert & = & C  \quad \text{for Bakhvalov-Shishkin meshes}
\end{eqnarray*}

$\bullet$ Shishkin mesh
\begin{equation*}
\psi_1(t) = e^{-2(1-2t)ln N}, \quad \,\  \psi_2(t) = e^{-4(1-t)ln N},
\end{equation*}

$\bullet$ Bakhvalov-Shishkin mesh
\begin{equation*}
\psi_1(t) = 1-2(1-N^{-1})(1-2t), \quad  \psi_2(t) = 1-4(1-N^{-1})(1-t).
\end{equation*}
The set of interior mesh points is denoted by $\Omega_\varepsilon^N
=\bar \Omega_\varepsilon^N \setminus \{x_{N/2}\}$. Also, for
the both meshes, on the coarse part $\Omega_S$ we have
\begin{equation*}
h_i\leq CN^{-1}.
\end{equation*}
It is well known that on the layer part of the Shishkin mesh \cite{rszar2004a}
\begin{equation*}
h_i\leq C \varepsilon N^{-1}\ln N
\end{equation*}
and of the Bakhvalov-Shishkin mesh we have
\begin{equation*}
h_i\leq \begin{cases} \frac{\tau_0}{\beta}\varepsilon N^{-1} \max {\mid \psi_1'\mid} \exp{(\frac{\beta}{\tau_0 \varepsilon}(d-x_{i-1}))},\quad i = N/4+1,...,N/2,\\
\frac{\tau_0}{\beta}\varepsilon N^{-1} \max {\mid \psi_2'\mid} \exp{(\frac{\beta}{\tau_0 \varepsilon}(1-x_{i-1}))},\quad i = 3N/4+1,...,N \end{cases}
\end{equation*}
and
\begin{equation*}
\frac{h_i}{\varepsilon} \leq CN^{-1}\max{\mid \varphi'\mid} \leq C.
\end{equation*}
\section{Interpolation Error}
Initially we consider the interpolation error in the maximum norm.
 Let $f \in C^2[x_{i-1},x_i]$ be arbitrary and $f^{I}$ a piecewise linear interpolant to $f$ on $ \Omega$.  Then from the classical theory,  we have
 \begin{equation*}
 \vert (f^{I} -f)(x) \vert \leq 2 \int_{x_{i-1}}^{x_i} \vert f''(t) \vert (t-x_{i-1}) dt.
 \end{equation*}
 Now we compute the interpolation error for the first component $u_1$.
 \begin{lemma}\label{lemmai}
 For the Shishkin mesh we have
 \begin{equation*}
 \vert u_i(x) - u_i^{I}(x) \vert \leq
 \begin{cases}
 C N^{-2}\ln^2 N, x\in \Omega_0\\
 C N^{-2},  x\in \Omega_S
 \end{cases}
 \end{equation*}
 and for the Bakhavalov-Shishkin mesh it holds
 \begin{equation*}
 \vert u_i(x) - u_i^{I}(x) \vert \leq C N^{-2}, x \in \Omega^- \cup \Omega^+, \quad i=1,2.
 \end{equation*}
 \end{lemma}
 \begin{proof}
We now give a proof for the case $i=1$ for the Shishkin mesh.  To prove the estimates we use the decomposition of solution as smooth and layer components and triangle inequality
\begin{equation}\label{ier1}
\vert (u_1-u_1^I)(x) \vert  \leq  \vert (v_1-v_1^I)(x) \vert + \vert (w_{11}-w_{11}^I)(x) \vert + \vert (w_{21}-w_{21}^I)(x) \vert.
\end{equation}
On Shishkin meshes,  let $x \in [x_{i-1},x_{i}] \subset \Omega^-\cap\Omega_S.$  Then  the first term of (\ref{ier1}) will be
\begin{eqnarray*}
\vert (v_1-v_1^I)(x) \vert & \leq & 2 \int_{x_{i-1}}^{x_i} \vert v_1''(t) \vert (t-x_{i-1}) dt\\
& \leq & 2C \int_{x_{i-1}}^{x_i} (t-x_{i-1}) dt \\
& \leq & 2C\frac{h_i^2}{2} \\
\vert (v_1-v_1^I)(x) \vert & \leq &  CN^{-2}.
\end{eqnarray*}
Again the second term of (\ref{ier1}) will be
\begin{eqnarray*}
\vert (w_{11}-w_{11}^I)(x) \vert & \leq & 2 \Vert w_{11}(x) \Vert_{L_{\infty}[x_{i-1},x_{i}]}\\
& \leq & C\max_{i} e^{\frac{-\beta (1-x_{i})}{\varepsilon}}\\
\vert (w_{11}-w_{11}^I)(x) \vert & \leq & CN^{-\tau_0}.
\end{eqnarray*}
To compute the last term of (\ref{ier1}),  we have
\begin{eqnarray*}
\vert (w_{21}-w_{21}^I)(x) \vert & \leq & 2 \Vert w_{21}(x) \Vert_{L_{\infty}[x_{i-1},x_{i}]}\\
& \leq & C \varepsilon \max_{i} e^{\frac{-\beta (d-x_{i})}{\varepsilon}}\\
& \leq & C N^{-1} \max_{i} e^{\frac{-\beta (d-x_{i})}{\varepsilon}}\\
\vert (w_{21}-w_{21}^I)(x) \vert & \leq & CN^{-1-\tau_0}.
\end{eqnarray*}
Now let $x \in [x_{i-1},x_{i}] \subset \Omega^-\cap\Omega_0$ we have
\begin{eqnarray*}
\vert (v_1-v_1^I)(x) \vert & \leq & 2 \int_{x_{i-1}}^{x_i} \vert v_1''(t) \vert (t-x_{i-1}) dt\\
& \leq & 2C \int_{x_{i-1}}^{x_i} (t-x_{i-1}) dt \\
& \leq & C \frac{h_i^2}{2} \\
& \leq & C (\varepsilon N^{-1}\ln N)^{2}
%\vert (v_1-v_1^I)(x) \vert & \leq &  CN^{-3}\ln^{2} N
\end{eqnarray*}
and also the second term on $\Omega_0$ will be
\begin{eqnarray*}
\vert (w_{11}-w_{11}^I)(x) \vert & \leq & 2 \Vert w_{11}(x) \Vert_{L_{\infty}[x_{i-1},x_{i}]}\\
& \leq & 2 \max_{i} e^{\frac{-\beta (1-x_{i})}{\varepsilon}}\\
%& \leq & C N^{-\tau_0}\\
\vert (w_{11}-w_{11}^I)(x) \vert & \leq & C N^{-\tau_0}.
\end{eqnarray*}
The last term on $\Omega_0$ will be
\begin{eqnarray*}
\vert (w_{21}-w_{21}^I)(x) \vert & \leq & 2 \Vert w_{21}(x) \Vert_{L_{\infty}[x_{i-1},x_{i}]}\\
& \leq & C \varepsilon \max_{i} e^{\frac{-\beta (d-x_{i})}{\varepsilon}}\\
& \leq & C N^{-1} \max_{i} e^{\frac{-\beta (d-x_{i})}{\varepsilon}}.\\
\vert (w_{21}-w_{21}^I)(x) \vert & \leq & CN^{-1-\tau_0}.
\end{eqnarray*}
Similarly  we will also obtain the same estimate on $x\in \Omega^+.$  From equation (\ref{ier1}),  hence the result.\\
On Bakhavalov-Shishkin mesh,  we follow the above similar procedure to obtain the result.
\end{proof}
 Now we consider the interpolation error of $\bar u$ in $L_2$-norm
 \begin{equation}\label{L2error}
 \Vert \bar u -\bar u^I \Vert_0 = [(\int_0^1 \vert u_1-u_1^I \vert^2 dx) + (\int_0^1 \vert u_2-u_2^I \vert^2 dx)]^{1/2}.
 \end{equation}
 \begin{lemma}\label{intpl3}
  For Shishkin mesh,  the interpolation error of $\bar u$ in $L_2$-norm is
\begin{equation*}
 \Vert \bar u - \bar u^I \Vert_0  \leq   CN^{-5/2} \ln^{5/2} N.
\end{equation*}
 \end{lemma}
\begin{proof}
Consider the  first component of equation (\ref{L2error})
\begin{eqnarray*}
 \int_0^1 \vert u_1-u_1^I \vert^2 dx & \leq & \sum_{i=1}^{N} \int_{x_{i-1}}^{x_i} \vert u_1-u_1^I \vert^2 dx\\
& \leq & \sum_{i=1}^{N} (CN^{-2}\ln^2 N)^2 h_i\\
& \leq & C_1(CN^{-2}\ln^2 N)^2(C \varepsilon \ln N) \\
 \Vert u_1 - u_1^I \Vert_0 & \leq & C N^{-5/2}\ln^{5/2} N.
\end{eqnarray*}
Similarly one can easily prove
\begin{equation*}
  \Vert u_2 - u_2^I \Vert_0 \leq CN^{-5/2} \ln^{5/2} N.
\end{equation*}
From (\ref{L2error}),  we have an estimate of $\bar u - \bar u^I $ in $L_2-$ norm
\begin{equation*}
\Vert \bar u - \bar u^I \Vert_0 \leq  CN^{-5/2} \ln^{5/2} N.
\end{equation*}
\end{proof}
\begin{lemma}\label{intpl4}
Let $\bar u$ and $\bar u^I$ be solution of (\ref{sys1}-\ref{sysbc}) and linear interpolant of $\bar u$ respectively.  Then we have
\begin{equation*}
\vert\vert\vert \bar u - \bar u^I \vert\vert\vert_{V_h}  \leq    C N^{-1} \ln^{3/2} N,  \text{for Shishkin meshes}
\end{equation*}
\end{lemma}
\begin{proof}
Since
\begin{equation*}
 \int_0^1 ((u_1-u_1^I)'(x))^2 dx = -\int_0^1 (u_1-u_1^I)(x)  u_1''(x) dx
\end{equation*}
therefore,  by Lemma \ref{lemmai} we conclude that
\begin{eqnarray*}
\int_0^1 \vert (u_1-u_1^I)'(x) \vert^2 dx & \leq & C \max_{x_i\in \Omega_{\varepsilon}^N} \vert(u_1-u_1^I)(x_i) \vert \sum_{i=1}^{N}\int_{x_{i-1}}^{x_i} u_1^{''}(x)dx \\
& \leq & C N^{-2} \ln^2 N  \sum_{i=1}^{N} \int_{x_{i-1}}^{x_i} (v_1^{''}(x) +w_{11}^{''}(x)+w_{21}^{''}(x))dx.
\end{eqnarray*}
then for the regular part of the solution we have
\begin{equation*}
\vert \sum_{i=1}^{N} \int_{x_{i-1}}^{x_i} v_1^{''}(x) \vert \leq  C(\varepsilon \ln N + 1)
\end{equation*}
and for the singular part
\begin{eqnarray*}
\vert \sum_{i=1}^{N} \int_{x_{i-1}}^{x_i} w_{11}^{''}(x) \vert &\leq & C \varepsilon^{-2} \sum_{i=1}^{N} \int_{x_{i-1}}^{x_i} e^{-\frac{\beta(1-x)}{\varepsilon}} dx\\
& \leq & C \varepsilon^{-1} [\sum_{i=\frac{N}{4}+1}^{\frac{N}{2}} [e^{-\frac{\beta(1-x)}{\varepsilon}}]_{x_{i-1}}^{x_i} + \sum_{i=\frac{3N}{4}+1}^{N} [e^{-\frac{\beta(1-x)}{\varepsilon}}]_{x_{i-1}}^{x_i}] + C \varepsilon^{-1} N^{1-\tau_0}\\
& \leq & C \varepsilon^{-1} \ln N.
\end{eqnarray*}
and
\begin{eqnarray*}
\vert \sum_{i=1}^{N} \int_{x_{i-1}}^{x_i} w_{21}^{''}(x) \vert &\leq & C \varepsilon^{-1} [\sum_{i=1}^{\frac{N}{2}} \int_{x_{i-1}}^{x_i} e^{-\frac{\beta(d-x)}{\varepsilon}} dx+\sum_{i=\frac{N}{2}+1}^{N} \int_{x_{i-1}}^{x_i} e^{-\frac{\beta(1-x)}{\varepsilon}} dx]\\
& \leq & C N^{1-\tau_0}
\end{eqnarray*}
Using the assumption $\tau_0 \geq 2$ and the above estimates we have
\begin{equation*}
\int_0^1 \vert (u_1-u_1^I)'(x) \vert^2 dx  \leq  C N^{-2}\ln^2 N (\varepsilon  \ln N + 1+\varepsilon^{-1} \ln N + N^{-1} )
\end{equation*}
 We also have similar result for $u_2$
\begin{equation*}
\int_0^1 \vert (u_2-u_2^I)'(x) \vert^2 dx  \leq C \varepsilon^{-1} N^{-2} \ln^3 N.
\end{equation*}
Now we combine the above results together
\begin{equation*}
\vert u_1 -u_1^I \vert_1^2 + \vert u_2 -u_2^I \vert_1^2 \leq C \varepsilon^{-1}N^{-2} \ln^3 N.
\end{equation*}
Here we have to compute the interpolation error of $\bar u$ in energy norm, that is,
 $\vert\vert\vert \bar u - \bar u^I \vert\vert\vert_{V_h}.$\\
 We have
\begin{align*}
\vert\vert\vert \bar u - \bar u^I \vert\vert\vert_{V_h}&  =  [\varepsilon (\vert u_1 -u_1^I \vert_1^2 + \vert u_2 -u_2^I \vert_1^2) + \sigma( \Vert u_1 -u_1^I \Vert_0^2 + \Vert u_2 -u_2^I \Vert_0^2)\\ & + \sum_{i=1}^{N}\int_{x_{i-1}}^{x_{i}}\delta_{1,i}b_1^2(x_i)((u_1-u_1^I)'(x))^2 dx\\
& + \sum_{i=1}^{N}\int_{x_{i-1}}^{x_{i}}\delta_{2,i}b_2^2(x_i)((u_2-u_2^I)'(x))^2 dx]^{1/2}.
\end{align*}
Now we have to estimate the following terms
\begin{align*}
\vert \sum_{i=1}^{N}\int_{x_{i-1}}^{x_{i}}\delta_{1,i}b_1^2(x_i)((u_1-u_1^I)'(x))^2 dx \vert & \leq C \mid \delta_{1,i} \mid  \sum_{i=1}^{N}\int_{x_{i-1}}^{x_{i}} \mid (u_1-u_1^I)'(x)^2 \mid dx \\
& \leq C  N^{-1}(\varepsilon^{-1}N^{-2}\ln^{3}N) \\
\vert \sum_{i=1}^{N}\int_{x_{i-1}}^{x_{i}}\delta_{1,i}b_1^2(x_i)((u_1-u_1^I)'(x))^2 dx \vert & \leq C  N^{-2}\ln^{3}N
\end{align*}
and also
\begin{align*}
\vert \sum_{i=1}^{N}\int_{x_{i-1}}^{x_{i}}\delta_{2,i}b_2^2(x_i)((u_2-u_2^I)'(x))^2 dx \vert  \leq C N^{-2}\ln^{3}N.
\end{align*}
Substituting these estimates, we have
\begin{align*}
\vert\vert\vert \bar u - \bar u^I \vert\vert\vert_{V_h} & \leq    [\varepsilon(C \varepsilon^{-1}N^{-2} \ln^3 N)+\sigma(CN^{-5/2} \ln^{5/2} N)^2+C N^{-3}\ln^{4}N+C N^{-2}\ln^{3}N]^{1/2}\\
& \leq [\varepsilon(C\varepsilon^{-1}N^{-2}\ln^{3}N) + \sigma(CN^{-5/2} \ln^{5/2} N)^2 + C N^{-2}\ln^{3}N]^{1/2}\\
& \leq C N^{-1}\ln^{3/2}N[ 1 + N^{-3}\ln^2 N + 1]^{1/2}\\
\vert\vert\vert \bar u - \bar u^I \vert\vert\vert_{V_h} & \leq    C N^{-1} \ln^{3/2} N.
\end{align*}
\end{proof}
\section{Error Estimate}
Now  we state the main theorem of this paper.
\begin{theorem}
Let $\bar u$ and $\bar u_h$ be solution of (\ref{sys1}-\ref{sysbc}) and (\ref{dissys}) respectively. Then we have
\begin{equation*}
\vert\vert\vert \bar u - \bar u_h \vert\vert\vert_{V_h} \leq \begin{cases} C N^{-1}\ln^{3/2} N, \quad  \text{for Shishkin mesh,}\\
C N^{-1}, \quad \text{for Bakhvalov-Shishkin mesh}.
\end{cases}
\end{equation*}
\end{theorem}
\begin{proof}
From the inequality (\ref{err1}), Lemmas (\ref{intpl2}), (\ref{intpl3}) and (\ref{intpl4}),  for Shishkin meshes  we have
\begin{eqnarray*}
\vert\vert\vert \bar u-\bar u_h \vert\vert\vert_{V_h} & \leq & CN^{-1} \ln^{3/2} N + C N^{-\frac{5}{2}}\ln^{\frac{5}{2}} N \\
& \leq & C N^{-1} \ln^{3/2} N.
\end{eqnarray*}
Similarly we prove the error estimates for Bakhvalov-Shishkin meshes.
\end{proof}

\section{Numerical Experiments}
 In this section we experimentally verify our theoretical results
 proved in the previous section.
\begin{example}\label{ex1}
  Consider the BVP
 \begin{eqnarray}\label{t1}
-\varepsilon u_1^{''} (x) + u_1'(x)+2u_1(x) - u_2(x) = f_1(x), \quad x \in \Omega^- \cup \Omega^+,\\
 -\varepsilon u_2^{''} (x) +u_2'(x)- u_1(x)  +2u_2(x) = f_2(x), \quad x \in \Omega^- \cup \Omega^+,
 \end{eqnarray}
 \begin{equation}\label{t1c}
  u_1(0)=0, \quad u_1(1)=0,\quad  u_2(0)=0,\quad  u_2(1)=0,
 \end{equation}
 where
 \begin{equation*}
 f_1(x)=\begin{cases} 1,\quad 0 \leq x \leq 0.5,\\
 -0.8,\quad 0.5 \leq x \leq 1
 \end{cases}
 \end{equation*}
 and
 \begin{equation*}
 f_2(x)=\begin{cases} -2.0,\quad 0 \leq x \leq 0.5,\\
 1.8,\quad 0.5 \leq x \leq 1
 \end{cases}
 \end{equation*}
\end{example}
For our tests, we take $\varepsilon = 2^{-18}$, which is sufficiently small to bring out the singularly perturbed nature of the problem. Now we define a maximum norm of $\bar u_h$ as \begin{equation*}
\parallel \bar u_h \parallel_{\infty} = \max \{ \max_{1\leq i \leq N-1} \{ \mid u_{1h}(x_i) \mid \},  \max_{1\leq i \leq N-1} \{ \mid  u_{2h}(x_i) \mid \} \}
\end{equation*}
We measure the accuracy in various norms and the rates of convergence $r^N$ are computed using the following formula:
\begin{equation*}
r^N = log_2(\dfrac{E^N}{E^{2N}}),
\end{equation*}
where
\begin{equation*}
 E^N = \begin{cases}
\parallel \bar u_h - \bar u_{2h}^I \parallel_{\infty}, \quad \text{for maximum norm,}\\
\parallel \bar u_h - \bar u_{2h}^I \parallel_{0}, \quad  \text{for} (L_{2}(\Omega))^2- \text{norm,}\\
\vert\vert\vert \bar u_h - \bar u_{2h}^I \vert\vert\vert_{V_h},\quad \text{ for discrete energy norm},
  \end{cases}
 \end{equation*}
  and $\bar u_h^I$ denotes the piecewise linear interpolant of $\bar U.$ \\
In Tables \ref{table1} and \ref{table2},  we present values of $E^N, r^N $ for the solution of the BVP (\ref{t1})-(\ref{t1c})  for Shishkin and Bakhvalov-Shishkin meshes respectively.
The Figures \ref{fig1}  and \ref{fig2} depict the numerical solution of the  BVP  (\ref{t1})-(\ref{t1c})  for Shishkin  mesh.
We compare the values of $E^N, r^N $ for the solution of the same BVP (\ref{t1})-(\ref{t1c})  for Shishkin mesh using the standard upwind scheme adopted \cite{tamram07a}.  From the tables,  we infer that the order of convergence is higher in the cases of maximum norm and $L_2-$ norm when compared with discrete energy norm as defined earlier.  Therefore the present method may yield better results.\\
\quad The numerical results are clear illustrations of the convergence estimates derived in the present paper for both the type of meshes.
\begin{remark}
It may be observed that the value of $\tau_0$ is taken as $\tau_0 \geq 2.$  From the above experimental results this condition seems to be essential.  Infact, it is found that if one takes the value $\tau_0< 2$ the order of convergence may not be $2.$
\end{remark}
\vspace{0.3cm}
\begin{table}[ht] \caption{\label{table1}\it{Values of $E^N$ and $r^N$ for
the solution of the BVP (\ref{t1}) - (\ref{t1c}) in different norms for Shishkin mesh.}} {\centering
\begin{tabular}{||c|c|c|c|c|c|c||} \hline \hline N & \multicolumn{2}{c|}{$\parallel \bar u_h - \bar u_h^I \parallel_{\infty}$} & \multicolumn{2}{c|}{$\parallel \bar u_h - \bar u_h^I \parallel_{0}$} &\multicolumn {2} {c|}{$\vert\vert\vert \bar u_h - \bar u_h^I \vert\vert\vert_{V_h}$} \\ \cline {2-7}
 &$E^N$ & $r^N$   & $E^N$ &$ r^N$  & $E^N$ & $r^N$
\\\hline
$ 32 $
 &2.3693e-01& 1.4253 &  1.0785e-02& 1.1113 & 2.7108e-01&  0.8742 \\
$ 64 $
 &8.8222e-02& 1.0592 &  4.9921e-03& 1.0447 & 1.4788e-01&  0.6716 \\
  $ 128 $
 &4.2337e-02& 0.9939&   2.4199e-03& 1.0176 & 9.2838e-02&  0.5973\\
    $ 256 $
 &2.1258e-02& 0.9986&   1.1953e-03& 1.0061 & 6.1517e-02&  0.5625\\
   $ 512 $
 &1.0639e-02& 1.0030&   5.9513e-04& 1.0011 & 4.1654e-02&  0.5620\\
   $ 1024 $
 &5.3085e-03& 1.0085&   2.9734e-04& 0.9990 & 2.8213e-02&  0.5522 \\
 $ 2048 $
 &2.6387e-03&      - &  1.4877e-04&      - & 1.9240e-02&      - \\
 \hline
 \end{tabular}\par}
 \end{table}
 \goodbreak\noindent
\vspace{0.3cm}

\vspace{0.3cm}
\begin{table}[ht] \caption{\label{table2}\it{Values of $E^N$ and $r^N$ for
the solution of the BVP (\ref{t1}) - (\ref{t1c}) in different norms for Bakhvalov-Shishkin mesh.}} {\centering
\begin{tabular}{||c|c|c|c|c|c|c||} \hline \hline N & \multicolumn{2}{c|}{$\parallel \bar u_h - \bar u_h^I \parallel_{\infty}$} & \multicolumn{2}{c|}{$\parallel \bar u_h - \bar u_h^I \parallel_{0}$} &\multicolumn {2} {c|}{$\vert\vert\vert \bar u_h - \bar u_h^I \vert\vert\vert_{V_h}$} \\ \cline {2-7}
 &$E^N$ & $r^N$   & $E^N$ &$ r^N$  & $E^N$ & $r^N$
\\\hline
$ 32 $
 &1.6550-01&  0.9811 &  1.1047e-02& 0.8554 & 2.7386e-01&  0.5465 \\
$ 64 $
 &8.3838e-02&  0.9865 & 4.9717e-03& 0.9120 & 1.8750e-01&  0.5304 \\
  $ 128 $
 &4.2313e-02&  0.9945&  2.3671e-03& 0.9535 & 1.2981e-01&  0.5194\\
    $ 256 $
 &2.1236e-02&  1.0001&  1.1551e-03& 0.9769 & 9.0558e-02&  0.5157\\
   $ 512 $
 &1.0617e-02& 1.0059&   5.7064e-04& 0.9874 & 6.3341e-02& 0.5192\\
   $ 1024 $
 &5.2870e-03& 1.0141&   2.8361e-04& 0.9940 & 4.4194e-02&  0.5322 \\
 $ 2048 $
 &2.6177e-03&      - &  1.4139e-04&      - & 3.0560e-02&      - \\
 \hline
 \end{tabular}\par}
 \end{table}
 \goodbreak\noindent
\vspace{0.3cm}
\begin{figure}
\centerline{
\begin{tabular}{cc}
\resizebox*{8cm}{!}{\includegraphics{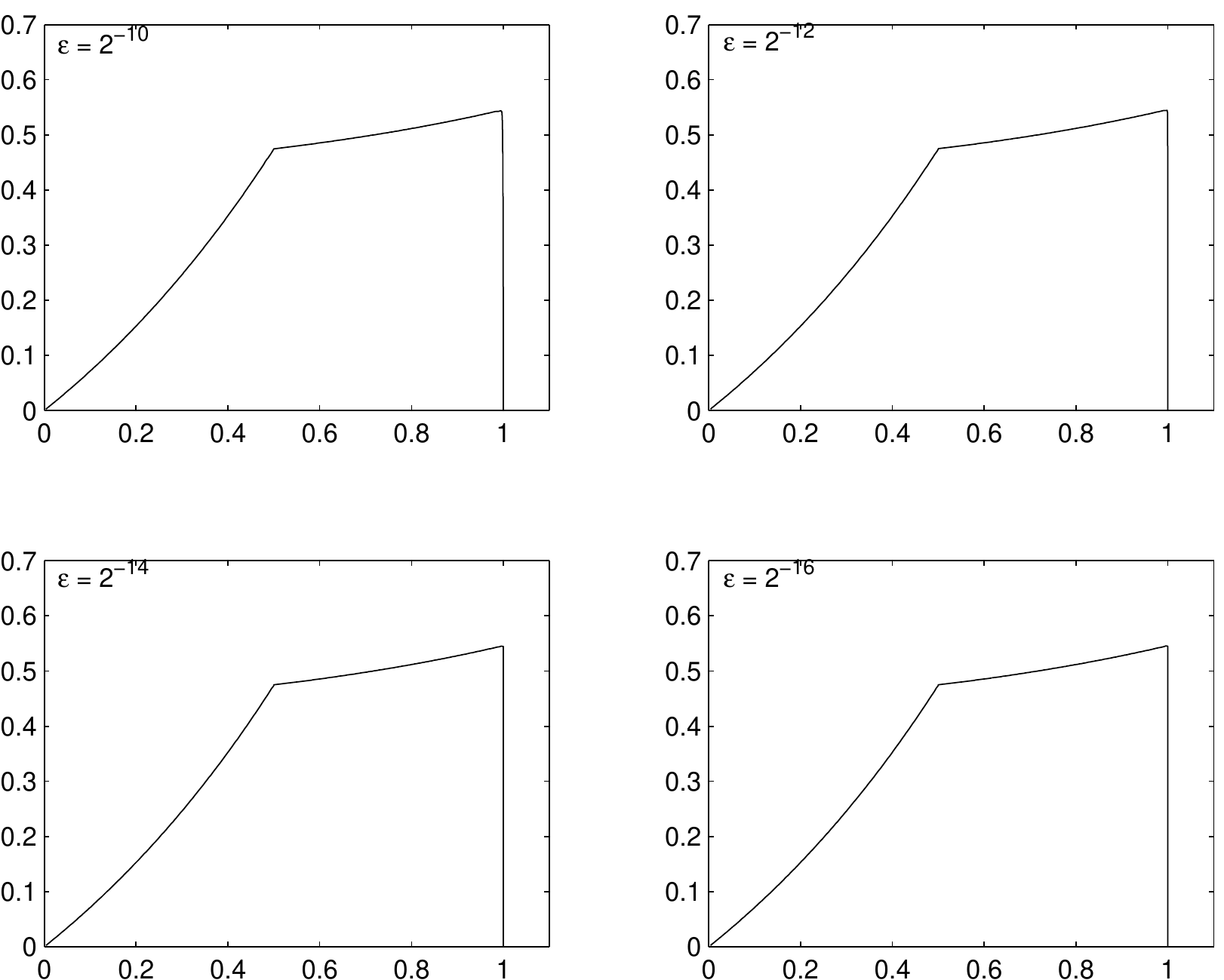}}
\end{tabular}
}
\caption{{\label{fig1}\it Graphs of the  numerical solution of the first component $u_{1h}$ of the BVP (\ref{t1})-(\ref{t1c}) for various values of $\varepsilon$ with $ N = 512.$}}
\end{figure}
\vspace{0.3cm}
\begin{figure}
\centerline{
\begin{tabular}{cc}
\resizebox*{8cm}{!}{\includegraphics{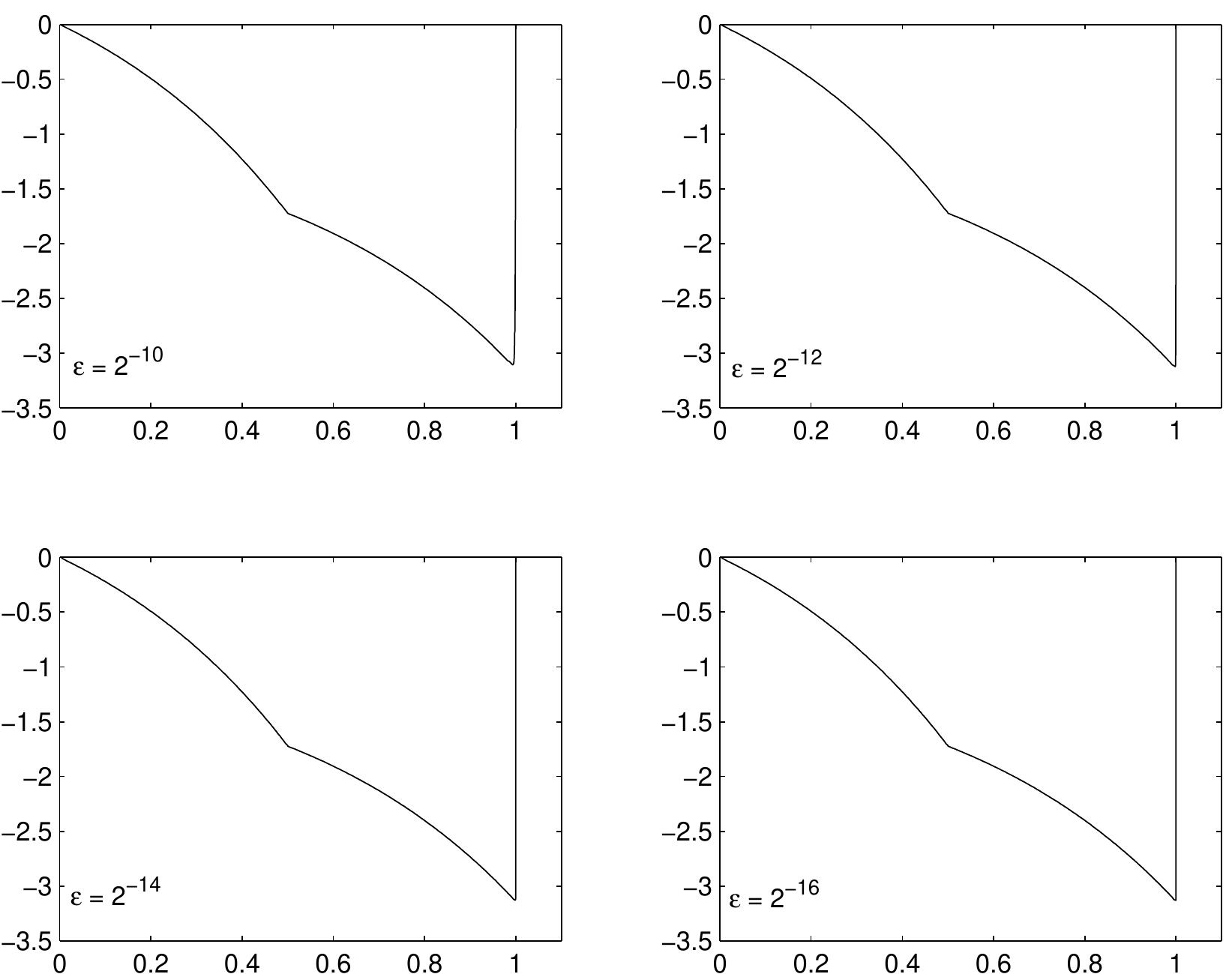}}
\end{tabular}
}
\caption{{\label{fig2}\it Graphs of the  numerical solution of the second component $u_{2h}$ of the BVP (\ref{t1})-(\ref{t1c}) for various values of $\varepsilon$ with $ N = 512.$}}
\end{figure}
\begin{figure}
\centerline{
\begin{tabular}{cc}
\resizebox*{8cm}{!}{\includegraphics{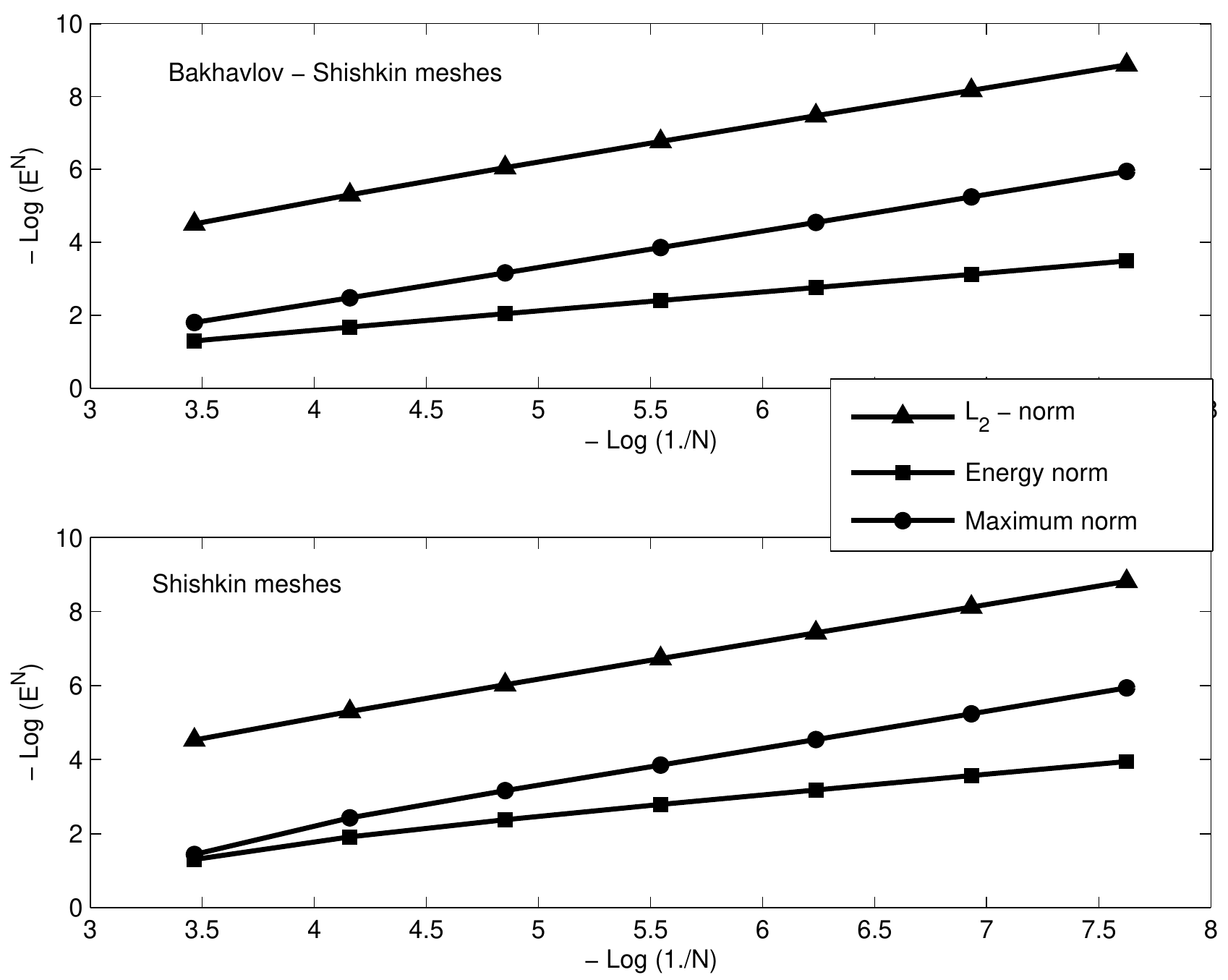}}
\end{tabular}
}
\caption{{\label{fig3}\it Plots of order of convergence for Example \ref{ex1} and  $\varepsilon = 2^{-18}$ in various norms.}}
\end{figure}
\goodbreak
% \bibliographystyle{elsart-num-sort}
% \bibliography{apnum}

\end{document}